\documentclass{amsart}
\usepackage{amsmath, amsthm, amssymb, amscd, mathrsfs, eucal, epsfig}
\usepackage{pinlabel}

\begin{document}

\font\bfit=cmbxti10

\newtheorem{theorem}{Theorem}[section]
\newtheorem{lemma}[theorem]{Lemma}
\newtheorem{sublemma}[theorem]{Sublemma}
\newtheorem{proposition}[theorem]{Proposition}
\newtheorem{corollary}[theorem]{Corollary}
\newtheorem{conjecture}[theorem]{Conjecture}
\newtheorem{question}[theorem]{Question}
\newtheorem{problem}[theorem]{Problem}
\newtheorem*{claim}{Claim}
\newtheorem*{criterion}{Criterion}
\newtheorem*{slww_inequality}{(Cor.~\ref{sl_inequality_lemma}) $\sl(w|w)$ inequality}
\newtheorem*{gamman_inequality}{(Prop.~\ref{gamma_n_inequality}) $\gamma_n$ inequality}
\newtheorem*{nilabelian}{(Thm.~\ref{nilpotent_implies_abelian})}
\newtheorem*{solabelian}{(Thm.~\ref{vanishing_scl_beta_2_abelian})}
\newtheorem*{gamman_duality}{(Thm.~\ref{duality_theorem}) $\gamma_n$-Duality theorem}
\newtheorem*{gamman_comparison}{(Thm.~\ref{comparison_theorem}) $\gamma_n$-Comparison theorem}
\newtheorem*{beta2_duality}{(Thm.~\ref{beta_2_duality_theorem}) $\beta_2$-Duality theorem}

\theoremstyle{definition}
\newtheorem{definition}[theorem]{Definition}
\newtheorem{construction}[theorem]{Construction}
\newtheorem{notation}[theorem]{Notation}

\theoremstyle{remark}
\newtheorem{remark}[theorem]{Remark}
\newtheorem{example}[theorem]{Example}

\numberwithin{equation}{subsection}

\def\Z{\mathbb Z}
\def\R{\mathbb R}
\def\N{\mathbb N}
\def\l{\textnormal{l}}
\def\sl{\textnormal{sl}}
\def\SL{\textnormal{SL}}

\def\CAT{\textnormal{CAT}}
\def\PL{\textnormal{PL}}
\def\cl{\textnormal{cl}}
\def\scl{\textnormal{scl}}
\def\area{\textnormal{area}}
\def\Aut{\textnormal{Aut}}
\def\seq(#1){{\lbrace #1 \rbrace}}
\def\fl(#1){{\lfloor #1 \rfloor}}

\def\id{\textnormal{id}}
\def\length{\textnormal{length}}

\title{Stable W-length}
\author{Danny Calegari}
\address{Department of Mathematics \\ Caltech \\
Pasadena CA, 91125}
\email{dannyc@its.caltech.edu}
\author{Dongping Zhuang}
\address{Department of Mathematics \\ Vanderbilt University \\ 
Nashville TN, 37240}
\email{dongping.zhuang@vanderbilt.edu}

\date{\today \quad version 0.13}

\begin{abstract}
We study stable $W$-length in groups, especially for $W$ equal to the $n$-fold commutator
$\gamma_n:=[x_1,[x_2,\cdots[x_{n-1},x_n]]\cdots]$. We prove that in any perfect group, for any $n\ge 2$ and any element $g$,
the stable commutator length of $g$ is at least as big as $2^{2-n}$ times 
the stable $\gamma_n$-length of $g$. We
also establish analogues of Bavard duality for words $\gamma_n$ and for $\beta_2:=[[x,y],[z,w]]$.
Our proofs make use of geometric properties of the asymptotic cones of verbal subgroups with respect
to bi-invariant metrics. In particular, we show that for suitable $W$,
these asymptotic cones contain certain subgroups
that are {\em normed vector spaces}.
\end{abstract}

\maketitle

\section{Introduction}

Geometric group theory aims to produce functors from the algebraic category of
groups and homomorphisms to geometric categories of spaces and structure-preserving maps.
The category of metric spaces and isometries does not have enough morphisms for many
applications, so one instead
typically studies functors from the category of groups and homomorphisms, to the category
of metric spaces and $1$-Lipschitz (i.e.\ distance decreasing) maps.

A rich source of such functors arise in the theory of {\em bounded cohomology}, introduced
systematically by Gromov in \cite{Gromov_volume}. In that theory, the metric spaces are
usually {\em normed vector spaces}, and the morphisms {\em bounded linear operators}.
Another rich source of such examples comes from the study of (conjugation-invariant)
{\em norms}; see e.g.\ \cite{Burago_Ivanov_Polterovich} for a discussion, and \cite{Kotschick} 
for an application to the theory of mapping class groups.

Such functors are often useful for the study of groups as dynamical objects, where the
functor ``geometrizes'' the group action, and allows one to obtain {\it a priori} control
of dynamical quantities from algebra. For example, {\em stable commutator length} (hereafter
$\scl$; see \cite{Calegari_scl} or \cite{Bavard} for an introduction)
has rich connections to $2$-dimensional dynamics, symplectic geometry, hyperbolic geometry,
and so on.

\medskip

A natural class of characteristic norms (those invariant under any automorphism), 
with good monotonicity properties, arise from the
theory of {\em words}. Given a subset $W$ of a free group $F$, a {\em $W$-word} in $G$ is the
image of some $w \in W$ under some homomorphism $\phi:F \to G$. The $W$-words in $G$ generate
a so-called {\em verbal subgroup} $G_W$ (see e.g.\ \cite{Neumann}), 
and the {\em $W$-length} of any $g\in G_W$ is defined
to be the smallest number of $W$-words and their inverses in $G$ whose product is equal to $g$.
For example, many authors study {\em square length}, which is $W$-length in the case $W$ consists
of the single word $x^2$. $W$-length (usually for $W$ consisting of a single word $w$) has
been intensively studied in finite groups, and recently some very strong theorems have been
obtained by Shalev and his collaborators (e.g.\ \cite{Shalev_waring, Shalev_waring2, LOST}).
See \cite{Segal} for a survey, and an introduction to some of these methods. However, with
some exceptions (notably \cite{Rhemtulla}), $W$-length has not been widely studied in (general)
infinite groups except in the special case of $W=\lbrace [x,y]\rbrace$ --- i.e.\ commutator
length. Part of the problem is that $W$-length seems to be such an unstructured quantity in
general, and is exceedingly hard to compute, or even to estimate --- even in finite groups!
Therefore in this paper we propose to study a suitable ``rationalization'' of this quantity,
namely {\em stable $W$-length}, where the stable $W$-length of an element $g\in G_W$ is defined
to be the limit of the quantity ($W$-length of $g^n$ divided by $n$) as $n \to \infty$. Our
aim is to generalize (to the extent that it is possible) some of what is known about stable
commutator length to more general classes of norms.

\medskip

Perhaps the most significant difficulty in generalizing the theory of (stable) commutator
length to more general stable lengths is that of {\em linearizing} the theory. The commutator
calculus exhibits an intimate connection between the algebraic theory of commutators and the
linear theory of ($2$-dimensional) homology. Commutators arise as boundaries in group homology,
and one may obtain a duality (known as {\em Bavard Duality}; see \cite{Bavard} and \cite{Calegari_scl},
Chapter 2) between stable commutator length and certain natural group
$1$-cocycles called {\em homogeneous quasimorphisms}. There is no natural homology theory
available for stable $W$-length for more general $W$, but one may obtain a generalization of the
theory for commutators by a {\em geometric} construction. Given a group $G$, one considers the
Cayley graph of $[G,G]$ taking as generators {\em all commutators in $G$}. This is a metric space
(a graph), and the group $[G,G]$ acts on itself by metric isometries. Because every commutator
has ``length $1$'' in this graph, the action of the group $[G,G]$ on itself is ``almost''
commutative. When one replaces the Cayley graph by its asymptotic cone (an infinite re-scaled
version), the limit becomes {\em exactly} commutative, and one obtains a normed vector space,
which can be identified with one of the natural spaces obtained from the homological approach.
It is this geometric construction that generalizes: given any $W$, one considers the Cayley
graph of $G_W$ with all $W$-words and their inverses as generators. Since $W$-words act with bounded
length in the Cayley graph, the rescaled asymptotic cone obeys the ``law'' that all $W$-words 
are {\em trivial}. With more work, one can obtain a certain subset of the asymptotic cone where,
for suitable $W$, the resulting group is actually {\em abelian}, and is in fact a normed vector
space where one can establish the analogue of Bavard duality. This program
is most successful when $W$ is an $n$-fold commutator $\gamma_n$, in which case we are able to
establish (for perfect groups $G$) a precise analog of Bavard duality, and to prove the existence
of {\em two-sided} estimates of stable $\gamma_n$-length in terms of stable commutator length.
This comparison theorem is a genuinely stable phenomenon, and does not hold for ``naive'' 
$\gamma_n$-length (for a precise statement of results, see \S~\ref{statement_of_results}).

\subsection{Statement of results}\label{statement_of_results}

We now give a precise statement of the main results in the paper.

In \S~\ref{basic_properties_section} we introduce terminology, and establish basic
properties. We show that there are various inequalities relating $W$-length and stable
$W$-length in various groups and for various different $W$. 
We use the notation (here and
elsewhere) of $\l(*|W)$ for $W$-length, and $\sl(*|W)$ for stable $W$-length, with $\cl(*)$ and
$\scl(*)$ for the special case of commutator length and stable commutator length respectively.

We give simple examples of groups
where stable $w$-length is nontrivial for essentially all words (free groups, hyperbolic groups)
and where stable $w$-length is trivial for all $w$ ($\SL(2,A)$ for certain rings of algebraic integers $A$).
We also derive some nontrivial estimates (usually upper bounds) on $\sl$ in free groups.
For example, 

\begin{slww_inequality}
For any $w$ there is an inequality
$$1/2 \le \frac {\scl(w)} {\scl(w) + 1/2} \le \sl(w|w) \le 1$$
\end{slww_inequality}

We strengthen the upper bound by an explicit construction, for certain natural classes of words.
Let $\gamma_n$ denote the {\em $n$-fold commutator}, so $\gamma_1=x$, $\gamma_2=[x,y]$,
$\gamma_3=[x,[y,z]]$ and so on. We obtain the estimate

\begin{gamman_inequality}
For any $n$ there is an inequality
$$\sl(\gamma_n|\gamma_n) \le 1 - 2^{1-n}$$
\end{gamman_inequality}

\S~\ref{cone_section} is the heart of the paper. To any group $G$ and any $W$, we study the asymptotic
geometry of the Cayley graph $C_W$ of $G_W$ with all $W$-words as generators. 
A limit of rescalings of $C_W$ gives rise to the {\em asymptotic cone} $\widehat{A}_W$.
The cone $\widehat{A}_W$
is itself a group, with a bi-invariant metric, obeying a nontrivial law. It contains a canonical
contractible subgroup $A_W$ which metrically encodes the values of stable $W$-length on $G_W$.
Our results are most definitive for certain specific $W$; in particular,

\begin{nilabelian}
Suppose $A_W$ is nilpotent (for example, if $W=\gamma_n$ for some $n$). 
Then it is a normed vector space (in particular it is abelian).
\end{nilabelian}

Moreover, if $\beta_2$ denotes the word $[[x,y],[z,w]]$,

\begin{solabelian}
Suppose in some group $G$ that stable commutator length vanishes identically. 
Then $A_W$ is abelian (and hence a normed vector space) for $W=\beta_2$.
\end{solabelian}

The fact that these asymptotic cones are normed vector spaces lets one study their geometry
(and thereby stable $W$-length) dually, via $1$-Lipschitz homomorphisms $A_W \to \R$. Motivated
by Bavard duality, we say a function $\phi:G \to \R$ is a {\em weak} $\gamma_n$-hoq (``homogeneous
quasimorphism'') if it is homogeneous (i.e.\ $\phi(g^n)=n\phi(g)$) and if there is a least non-negative
real number $D(\phi)$ (called the {\em defect}) such that for any $g,h \in G$ there is an inequality 
$$|\phi(g) + \phi(h) - \phi(gh)| \le D(\phi)\min(\l(g|\gamma_{n-1}),\l(h|\gamma_{n-1}),\l(gh|\gamma_{n-1}))$$

\begin{gamman_duality}
For any perfect group $G$, and any $g \in G$ there is an inequality
$$\sup_\phi \phi(g)/D(\phi) \ge \sl(g|\gamma_n) \ge \sup_\phi \phi(g)/2D(\phi)$$
where the supremum is taken over all weak $\gamma_n$-hoqs.
\end{gamman_duality}

From this we derive the following rather surprising estimate, comparing stable commutator length
with stable $\gamma_n$ length in any perfect group:

\begin{gamman_comparison}
For any perfect group $G$, for any $n\ge 2$ and for any $g\in G$ there is an inequality
$$2^{n-2}\scl(g) \ge \sl(g|\gamma_n) \ge \scl(g)$$
\end{gamman_comparison}

Together the $\gamma_n$-Comparison theorem and the $\gamma_n$-Duality theorem show that stable
commutator length in a perfect group $G$ can be bounded from below by weak $\gamma_n$-hoqs. It would
be interesting to try to find naturally occurring examples of such functions arising from realizations
of $G$ as a group of automorphisms of some geometric object.

\medskip

In a similar vein, we say $\phi:G\to\R$ is a {\em weak} $\beta_2$-hoq if it is homogeneous, and
if there is a least non-negative
real number $D(\phi)$ (called the {\em defect}) such that for any $g,h \in G$ there is an inequality 
$$|\phi(g) + \phi(h) - \phi(gh)| \le D(\phi)\min(\l(g|\beta_2),\l(h|\beta_2),\l(gh|\beta_2))$$

\begin{beta2_duality}
For any perfect group $G$ in which $\scl$ vanishes identically, and for
any $g \in G$ there is an inequality
$$\sup_\phi 2\phi(g)/D(\phi) \ge \sl(g|\beta_2) \ge \sup_\phi \phi(g)/2D(\phi)$$
where the supremum is taken over all weak $\beta_2$-hoqs.
\end{beta2_duality}
\noindent Note that stable commutator length is known to vanish identically in many
important classes of groups, including
\begin{itemize}
\item{amenable groups}
\item{lattices in higher rank Lie groups \cite{Burger_Monod}}
\item{groups of $\PL$ homeomorphisms of the interval \cite{Calegari_PL}}
\item{groups that satisfy a law \cite{Calegari_laws}}
\end{itemize}
It is not known whether stable $W$-length for various $W$ vanish
in these classes of groups. One of the aims of this paper is to develop
tools to approach this question.

\medskip

Finally, in \S~\ref{grope_section}, we give an elementary construction, using hyperbolic
geometry, of a class of elements $w_n$ in free groups of various ranks for which 
$\scl(w_n)=1/2$ and $\sl(w_n|\gamma_3)=1$,
but for which $\cl(w_n)=1$ and $\l(w_n|\gamma_3)\ge 2n/3$.

\subsection{Acknowledgments}

We would like to thank Frank Calegari, Benson Farb, Denis Osin and Dan Segal
for helpful comments. We would also like to thank the anonymous referee for a very careful reading
and for catching several errors. Danny Calegari was supported by NSF grant DMS 1005246.

\section{Basic properties}\label{basic_properties_section}

The purpose of this section is to standardize definitions and notation, and to survey some
elementary constructions in the theory.

\begin{notation}
We use the notation $[x,y]$ for $xyx^{-1}y^{-1}$ and $x^y$ for $yxy^{-1}$. With this convention, 
exponentiation obeys $(x^y)^z = x^{zy}$. We also use the notation $x^*$ for some (unspecified)
conjugate of $x$.
\end{notation}

The following identities, though elementary, are useful.

\begin{lemma}\label{elementary_identities}
The following identities hold (the letters denote free generators):
\begin{enumerate}
\item{$[x,y] = [y^x,x^{-1}]$}
\item{$[x,y]^{-1} = [y,x] = [x^y,y^{-1}]$}
\item{$[x,yz]=[x,y][x,z]^y=[x,y][x,z][[z,x],y]$}
\item{$[xy,z]=[x,z][y,z]^{x^z}=[x,z][y,z][[z,y],x^z]$}
\item{$[[y,x],z^x][[x,z],y^z][[z,y],x^y]=1$}
\end{enumerate}
\end{lemma}

Bullet (5) is known as the {\em Hall-Witt identity}.

\subsection{Verbal subgroups}

\begin{definition}[\protect{$W$-word}]
Let $F$ be a free group, and $W$ a subset of $F$. A {\em $W$-word} in a group $G$ is the image of some
$w \in W$ under a homomorphism $F \to G$.
\end{definition}

\begin{definition}[verbal subgroup]
Let $G$ be a group. A {\em verbal subgroup} of $G$ is the subgroup generated by the $W$-words, for
some $W \subset F$. The $W$-verbal subgroup of $G$ is denoted $G_W$.

The group $G$ is said to {\em obey the law} $W$ if $G_W$ is trivial. A group is said to 
{\em obey a law} if $G_W$ is trivial for some nonempty $W$.
\end{definition}

For any sets $V,W \subset F$ with $F_V=F_W$ (for example, if $W \subset V \subset F_W$)
and for any group $G$, the verbal subgroups $G_V$ and $G_W$ are equal, though the 
$V$-words and the $W$-words in $G$ might not be equal {\em as sets}.

Since the image of a $W$-word under any automorphism is also a $W$-word, verbal subgroups are 
characteristic (i.e.\ invariant under every automorphism). In particular, a conjugate of a $W$-word
is a $W$-word. In a free group, it is tautologically true that characteristic subgroups are verbal.

\medskip

In the sequel we are interested in verbal subgroups $G_W$ where $W$ is a single element of $F$.
For the sake of legibility, we use a lower case $w$ to denote a single element of $F$, and likewise
denote the $w$-subgroup of $G$ by $G_w$. In fact, we are interested in certain explicit words
$w$. The main $w$ of interest in this paper are the {\em $n$-fold commutators}.

\begin{definition}
Here and in the sequel, we let $\gamma_n$ denote the $n$-fold commutator of the 
generators of $F_n$. That is,
$$\gamma_1 = x, \quad \gamma_2 = [x,y], \quad \gamma_3 = [x,[y,z]], \quad \gamma_4 =[x,[y,[z,w]]]$$
and so on.
\end{definition}

\begin{definition}
A word $w \in F$ is {\em reflexive} if $w^{-1}$ is a $w$-word in $F$.
\end{definition}

\begin{lemma}
The words $\gamma_n$ are all reflexive. Moreover, any nested bracket of $n$ elements is a
$\gamma_n$ word.
\end{lemma}
\begin{proof}
The second statement means, for example, that an expression like $[[x,y],z]$ is a $\gamma_3$ word,
that $[x,[[ [w,[u,v]], z],y]]$ is a $\gamma_5$ word, and so on. This and the reflexivity of $\gamma_n$
follow inductively from Lemma~\ref{elementary_identities}, bullet (2).
\end{proof}

For any $n$ and any $G$, the subgroups $G_{\gamma_n}$ are usually denoted $G_n$ and called
the {\em lower central series}. The successive quotients $G/G_n$ are (universal) $(n-1)$-step
nilpotent quotients of $G$ (so that for instance $G/G_2$ is abelian, $G/G_3$ is $2$-step nilpotent, and
so on).

\begin{definition}
Here and in the sequel, we define $\beta_1=[x,y]$, and for $n>1$ we let $\beta_n$ denote the 
commutator of two copies of $\beta_{n-1}$ (in different generating sets). Hence
$$\beta_1=[x,y], \quad \beta_2=[[x,y],[z,w]], \quad \beta_3=[[[x,y],[z,w]],[[s,t],[u,v]]]$$
and so on.
\end{definition}

\begin{lemma}
The words $\beta_n$ are reflexive.
\end{lemma}
\begin{proof}
This follows inductively from Lemma~\ref{elementary_identities}, bullet (2).
\end{proof}

The subgroups $G_{\beta_n}$ are usually denoted $G^{(n)}$ and called the {\em derived series}. The
successive quotients $G/G^{(n)}$ are (universal) $n$-step solvable quotients of $G$. Note 
that $\beta_1=\gamma_2$.

\subsection{\protect{$W$-length and $w$-length}}

\begin{definition}[\protect{$W$-length}]
Let $W \subset F$ be given.
Let $G$ be a group, and let $G_W$ be the subgroup generated by $W$-words. If $g \in G_W$, the
{\em $W$-length} of $g$, denoted $\l(g|W)$, 
is the minimum number of $W$-words and their inverses in $G$ whose product is $g$, and the
{\em stable $W$-length} of $g$, denoted $\sl(g|W)$, is the limit
$$\sl(g|W) = \lim_{n \to \infty} \frac {\l(g^n|W)} {n}$$
\end{definition}

If $W$ consists of a single element, we usually denote it by a lower case $w$, and define
$w$-length and stable $w$-length. We are typically concerned with this case in the sequel.

If $w$ is reflexive (which shall usually be the case in the sequel), the inverse of
a $w$-word is a $w$-word. By abuse of notation we will sometimes refer 
to both $w$-words and their inverses as $w$-words. The
meaning should be clear from context in each case.

\begin{notation}
If we need to emphasize that $g$ is in specific group $G$, we use the notation $\l_G(g|W)$ and
$\sl_G(g|W)$.

In the special case that $w=[x,y]$ in the free group generated by $x$ and $y$, we refer to $w$-length
as {\em commutator length}, and stable $w$-length as {\em stable commutator length}, and denote them by
$\cl(\cdot)$ and $\scl(\cdot)$ respectively. 
\end{notation}

\begin{notation}\label{seq_notation}
We abbreviate the product of $n$ (arbitrary) $w$-words by $w^\seq(n)$ (or $W^\seq(n)$ for $W$-words).
This should not be confused with $w^n$, the $n$th power of a specific word $w$.
\end{notation}

\begin{lemma}\label{commutator_or_trivial}
Suppose $w$ is not in $[F,F]$. Then $\sl(\cdot|w)$ is identically zero in every group.
\end{lemma}
\begin{proof}
Since $w$ is not in $[F,F]$, there is a surjective homomorphism $F \to \Z$ sending $w$ to some nonzero
$n$. Without loss of generality, we may assume $n>0$. But then in any group $\l(g^{nm}|w) \le 1$
so $\sl(g|w)=0$.
\end{proof}

\begin{lemma}\label{monotonicity}
The functions $\l(\cdot|W)$ and $\sl(\cdot|W)$ are characteristic, and their values do 
not increase under homomorphisms.
\end{lemma}
\begin{proof}
The group $\Aut(G)$ permutes the canonical generators of $G_W$, proving the first claim. The image
of a $W$-word under any homomorphism is a $W$-word, proving the second claim.
\end{proof}

By convention, we set $\sl(g|W) = \sl(g^n|W)/n$ whenever $g^n \in G_W$ (even if $g$ is
not necessarily in $G_W$).

\subsection{Inequalities}

If $w$ and $v$ are contained in free groups $F$ and $F'$ respectively, by abuse of notation we think of
$w$ and $v$ as being contained in the single free group $F*F'$. By convention therefore we think of
all abstract words as being contained in a single (infinitely generated) free group, which we denote
hereafter by $F$.

\begin{lemma}[fundamental inequality]\label{fundamental_inequality_lemma}
For any $v,w \in F$ and any $g \in G$ there is an inequality
$$\l_G(g|v) \le \l_G(g|w)\l_F(w|v)$$
\end{lemma}
\begin{proof}
Write $g$ minimally as a product of $w$-word or their inverses, and then rewrite each of these as
a minimal product of $v$-words or their inverses. This gives an expression for $g$ as a product of
$v$-words and their inverses.
\end{proof}

Note that for this lemma to make sense we must have $g \in G_w$ and $w \in F_v$.
In the case $v=[x,y]$ we obtain $\cl_G(g) \le \l_G(g|w)\cl_F(w)$. 
In the case of stable commutator length, one obtains a better estimate

\begin{lemma}[$\scl$ inequality]\label{scl_inequality_lemma}
For any $w \in F$ and any $g \in G$ and $n \in \Z$ there is an inequality
$$\scl_G(g^n) \le \l_G(g^n|w)\scl_F(w) + \tfrac 1 2 (\l_G(g^n|w) - 1)$$
Moreover,
$$\scl_G(g) \le \sl_G(g|w)(\scl_F(w) + \tfrac 1 2)$$
\end{lemma}
\begin{proof}
Express $g^n$ as a product $w_1 \cdots w_m$ of $w$-words or their inverses, where $m=\l_G(g^n|w)$.
For any even $k$ we have 
$$g^{nk} = (w_1\cdots w_m)^k = w_1^k\cdots w_m^k v$$
where $v$ is a product of at most $(m-1)k/2$ commutators (see e.g.\ \cite{Calegari_scl}, Lemma~2.24
for the case $m=2$, and apply induction). Taking the limit as $k \to \infty$ we obtain the first inequality.

Taking the limit $n \to \infty$ and stabilizing gives us the second inequality.
\end{proof}

\begin{corollary}[$\sl(w|w)$ inequality]\label{sl_inequality_lemma}
For any $w \in F$ there is an inequality
$$1/2 \le \frac {\scl(w)} {\scl(w) + 1/2} \le \sl(w|w) \le 1$$
\end{corollary}
\begin{proof}
The first inequality follows from the fact that $\scl(w) \ge \tfrac 1 2$ for any nontrivial $w$ in a free group
(see e.g.\ \cite{Calegari_scl}, Theorem~4.111). The second inequality follows from Lemma~\ref{scl_inequality_lemma}
by setting $g=w$. The last inequality follows from the tautological expression of $w^n$ as a product of
$n$ copies of $w$.
\end{proof}

A natural question to ask is the following:

\begin{question}[strict inequality]
For which $w \in F$ is there is a strict inequality
$$\sl(w|w) < 1 \text{ ?}$$
\end{question}

This question has a positive answer for several interesting classes of words.

\begin{example}
For each $g$, let $x_1,y_1,\cdots,x_g,y_g$ be free generators, and let $w_g=[x_1,y_1]\cdots[x_g,y_g]$.
Since $\scl(w_g) = g- \tfrac 1 2$ (see e.g.\/ Bavard \cite{Bavard})
we obtain a lower bound $\sl(w_g|w_g) \ge 1-1/2g$. On the other hand, an
expression of $w_g^n$ as a product of $n(g-1/2) + o(n)$ commutators is also an expression as a product of
$n(g- \tfrac 1 2)/g + o(n)$ $w_g$-words. So
$$\sl([x_1,y_1]\cdots[x_g,y_g]|[x_1,y_1]\cdots[x_g,y_g]) = 1-1/2g$$
\end{example}

From the definition, $\sl(g|w) \le \l(g^n|w)/n$ for any $n$. On the other hand, this inequality can often
be definitely improved:

\begin{lemma}\label{stable_promotion_inequality}
For any $w \in F$ and any $g \in G$, there is an inequality
$$\sl_G(g|w) \le (\l(g^n|w)-1 + \sl(w|w))/n$$
\end{lemma}
\begin{proof}
Suppose $g^n = w_1\cdots w_m$ where $m=\l(g^n|w)$. For any $k$ we can write
$$g^{nk} = w_1\cdots w_m w_1\cdots w_m \cdots w_1 \cdots w_m \text{ ($mk$ terms)}$$
$$ = w_1^k w_1'w_2' \cdots w_{(m-1)k}'$$
where each $w_i'$ is a $w$-word (since $x y = y (y^{-1}xy)$, and any conjugate of a $w$-word or
its inverse is a $w$-word or its inverse). In particular, if $k$ is large, then $w_1^k$ can be expressed
as a product of $k\,\sl(w|w) + o(k)$ $w$-words and therefore $g^{nk}$ can be expressed as a product of
$k\,\sl(w|w) + (m-1)k + o(k)$ $w$-words. Taking $k$ large, we obtain the desired inequality.
\end{proof}

\begin{corollary}
If $\l(w^n|w)=m$ then $\sl(w|w) \le (m-1)/(n-1)$.
\end{corollary}
\begin{proof}
Substitute $w=g$ in Lemma~\ref{stable_promotion_inequality}.
\end{proof}

We end this section with a couple of examples, illustrating the range of possibilities one
can expect in certain specific classes of groups.

\begin{example}[Free group]
If $F$ is a free group, $\scl(g)$ is positive for every $g\in [F,F]$ (in fact, $\scl(g)\ge 1/2$).
So for every $w$ with $F_w \subset [F,F]$ we get $\sl_F(g|w)\ge \scl_F(g)/(\scl_F(w)+1/2) > 0$.
In other words, for every $w$ for which $\sl(*|w)$ is not identically zero in every group, it is
positive on every element in a free group. A similar phenomenon holds in groups for which
$\scl$ is typically positive, such as hyperbolic groups (see e.g.\ \cite{Calegari_scl}, Chapter~3).
\end{example}

\begin{example}[$\SL(2,A)$]
Carter-Keller-Paige \cite{Carter_Keller_Paige} Thm.~6.1. show that if $A$ is the ring of
integers in a number field $K$ containing infinitely many units, and if $T$ is an element of 
$\SL(2,A)$ which is not a scalar matrix (i.e.\ not of the form $\lambda\cdot\id$) then
$\SL(2,A)$ has a finite index normal subgroup which is boundedly generated by conjugates of
$T$. That is, in this finite index normal subgroup, every element can be written as a bounded
number of conjugates of $T$.

For any $w$, one can find $T$ in $\SL(2,A)$ which is a $w$-word but not a scalar matrix (to see
this, observe that $\SL(2,A)$ contains $\SL(2,\Z)$ which contains a free group, none of whose
nontrivial elements is a scalar matrix). It follows that $\SL(2,A)$ has a finite index subgroup
in which every element has finite $w$-length. Consequently $\sl(*|w)$ vanishes identically in
$\SL(2,A)$ for every $w$.

A similar phenomenon holds in groups such as $\SL(n,\Z)$ when $n\ge 3$.
\end{example}

\subsection{Culler's identity}

Culler \cite{Culler}, in his analysis of commutator length in free groups, discovered the beautiful identity
$$[x,y]^3 = [y^x,x^{y^{-1}}x^{-2}][x^{y^{-1}},y^2]$$
This identity expresses algebraically the geometric fact that a once-punctured torus admits an (irregular) cover
of degree $3$ with one boundary component. We can get a lot of mileage out of this identity; this is
not really for a very specific reason, rather because this is the simplest identity which certifies
that $\scl([x,y])<1$, and we are able to ``bootstrap'' this identity to show $\sl(w|w)<1$ for many $w$.

\begin{example}[\protect{$[x,y^n]$}]\label{[x,y^n]example}
If we set $w=[x,y^n]$ then we obtain $\l(w^3|w) \le 2$ and therefore $\sl(w|w) \le \tfrac 1 2$. Together with
Corollary~\ref{sl_inequality_lemma} this implies
$$\sl([x,y^n]|[x,y^n])=\tfrac 1 2$$
\end{example}

\begin{example}[\protect{$[x,y]^2$}]\label{[x,y]^2example}
If we denote $a=[y^x,x^{y^{-1}}x^{-2}]$ and $b=[x^{y^{-1}},y^2]$ then Culler's identity says
$$([x,y]^2)^6 = ([x,y]^3)^4 = abababab$$
$$ = aabbcaabbc \text{ for }c=[b^{-1},b^{-1}a^{-1}b^{-1}]$$
$$ = a^2b^2c^2(a')^2(b')^2$$
for suitable commutators $a',b'$.
Hence $\l(([x,y]^2)^6|[x,y]^2) \le 5$ and therefore
$$\tfrac 2 3 \le \sl([x,y]^2|[x,y]^2) \le \tfrac 4 5$$
\end{example}

\subsection{\protect{$\gamma_n$} inequalities}

Generalizations of Culler's identity can be obtained from a uniform topological argument.

\begin{lemma}[generalized Culler identity]\label{generalized_Culler_identity}
For all $x$, for any $n$ and any $y$, there is an identity of the form
$$[x,y]^{2n+1} = [*,y^{n+1}]^*\cdot\text{product of }n\text{ words of the form }[*,y]^* $$
\end{lemma}
\begin{proof}
Let $F = \langle x,y \rangle$. For each $n$, we construct a permutation representation $\rho_n:F \to S_{2n+1}$.
We think of an element of $S_{2n+1}$ as a bijective function from $\lbrace 0,1,\cdots, 2n\rbrace$ to
itself, and denote such a function by an ordered string of values. With this convention, we define
$$\rho_n(x) = 2n\; 2n-1\; 2n-2 \cdots 1\; 0$$
$$\rho_n(y) = 1\; 2\; 3 \cdots n\; 0\; n+1 \cdots 2n$$
In words, $\rho_n(y)$ cycles the first $n+1$ elements and fixes the rest, 
while $\rho_n(x)$ is the involution which reverses the order.

If we think of $F$ as the fundamental group of a once-punctured torus $S$, then $x$ and $y$ can be represented
by embedded essential loops $\alpha$ and $\beta$ intersecting in one point. Since $\beta$ 
is homologically essential, it is 
non-separating. Let $S'$ be the $2n+1$-fold cover of $S$ whose monodromy representation is $\rho_n$.
By the definition of $\rho_n$, the permutation $\rho_n([x,y])$ is a single $2n+1$-cycle, and therefore $\partial S'$
consists of a single boundary component, mapping to $\partial S$ with degree $2n+1$ under the covering
projection. 

Likewise, $\beta$ has $n+1$ preimages; one (call it $\beta_0$) maps to $\beta$ with degree $n+1$. The
others $\beta_i$ for $1\le i \le n$ map with degree $1$. We claim that there are disjoint, embedded
curves $\alpha_i$ in $S'$ so that each $\alpha_i$ and $\beta_i$ intersect transversely in one point, and
$\alpha_i$ does not intersect $\beta_j$ for $i \ne j$. Evidently the lemma follows from this.

To prove this claim we first show that $\cup \beta_i$ is non-separating. Let $T$
be a component of $S' - \cup \beta_i$. Then $T$ covers $S-\beta$, and consequently $T\cap \partial S'$ is
nonempty. But this implies $T\cap \partial S' = \partial S'$, since $\partial S'$ is connected.
Since $T$ was arbitrary, every component of $S'-\cup \beta_i$ contains $\partial S'$, and therefore
$S'-\cup\beta_i$ is connected and $\cup \beta_i$ is non-separating, as claimed.

It follows that $S' - \cup_i \beta_i$ is
connected, and therefore (by reason of $\chi$) is homeomorphic to a sphere with $2n+3$ holes. We can therefore
find $n+1$ disjoint properly embedded arcs in $S' - \cup_i \beta_i$, each of which runs 
between the two boundary components corresponding to a single $\beta_i$. This proves the claim, and therefore
the lemma.
\end{proof}

As a corollary, we obtain the following inequality:

\begin{proposition}[\protect{$[x,w]$ inequality}]\label{[x,w]_inequality}
For any word $w$, the word $[x,w]$ (where $x$ is an extra free generator) satisfies
$$\l([x,w]^{2n+1}|[x,w])\le n+ \l(w^{n+1}|w)$$ and consequently
$$\sl([x,w]|[x,w]) \le \frac {1 + \sl(w|w)} {2}$$
\end{proposition}
\begin{proof}
The proposition follows from Lemma~\ref{generalized_Culler_identity}, together with the observation that
$[x,w^{n+1}]$ can be written as a product of $\l(w^{n+1}|w)$ $[x,w]$-words, by bullet (3) of Lemma~\ref{elementary_identities}.
\end{proof}

As an important special case, one deduces:

\begin{proposition}[\protect{$\gamma_n$ inequality}]\label{gamma_n_inequality}
For any $n$ there is an inequality
$$\sl(\gamma_n|\gamma_n) \le 1 - 2^{1-n}$$
\end{proposition}
\begin{proof}
For $n=1$ this follows from $\sl(x|x)=0$, a special case of Lemma~\ref{commutator_or_trivial}. Then
the result follows from Proposition~\ref{[x,w]_inequality} and induction on $n$.
\end{proof}

\begin{question}
Is the estimate in Proposition~\ref{gamma_n_inequality} sharp?
\end{question}

\section{Geometry of \protect{$G_W$}}\label{cone_section}

In this section we begin a more systematic study of stable $W$-length, from a geometric point of view.
The main tool is the geometry of Cayley graphs, and their asymptotic cones. A verbal subgroup of any group
admits a tautological, characteristic, bi-invariant metric, and it is the asymptotic geometry of this
metric that gives us insight into stable $W$-length.

\subsection{Cayley graph \protect{$C_W$}}

Let $G$ be a group, and $G_W$ its verbal subgroup associated to some subset $W \subset F$. If $w$
is a single element, we assume it is reflexive. Otherwise, we assume that $W$ is symmetric; 
i.e.\ $W=W^{-1}$.

Let $C_W(G)$ (or just $C_W$ or $C$ if $G$ or $W$ are understood) denote the Cayley graph of $G_W$
with all $W$-words in $G$ as generators. We give $C$ the structure of a (path) metric space,
by declaring that each edge has length $1$. This induces a metric on $G_W$ which we call the
$W${\em -metric}, where the distance
from $g$ to $h$ is the $W$-length of $g^{-1}h$.

\begin{lemma}\label{biinvariance}
The $W$-metric is both left and right invariant for the action of $G_W$ on itself. Furthermore,
$\Aut(G)$ acts on $G_W$ by isometries, and any homomorphism $G \to H$ induces a $1$-Lipschitz
(simplicial) map from $C_W(G)$ to $C_W(H)$.
\end{lemma}
\begin{proof}
For any $f,g,h \in G_W$ we have 
$$d(fg,fh) = \l(g^{-1}h|W) = d(g,h)$$
and
$$d(gf,hf) = \l((g^{-1}h)^{f^{-1}}|W) = \l(g^{-1}h|W) = d(g,h)$$
The latter two properties are restatements of Lemma~\ref{monotonicity} in geometric language.
\end{proof}

From the construction, the stable $W$-length of an element $g \in G_W$ is just the {\em translation length}
$\tau(g)$ of $g$ on $G_W$; i.e.\ the limit $\tau(g) = \lim_{n \to \infty} d(\id,g^n)/n$.

The large-scale geometry of commutator subgroups (i.e.\ the case $w =[x,y]$)
was studied in \cite{Calegari_Zhuang}.

\subsection{Asymptotic cone}

Though $C_W$ is an interesting geometric object in its own right, stable $W$-length can be more
easily studied in an object derived from $C_W$, namely the {\em asymptotic cone}; see e.g.\ 
Gromov \cite{Gromov_asymptotic} Chapter~2 (especially p.~36) 
for an introduction to asymptotic cones in group theory and their properties.
The construction of an asymptotic cone depends (in general) on a highly non-constructive choice,
namely the choice of a non-principal ultrafilter. However, we will shortly restrict attention to a
subset of the asymptotic cone whose geometry does not depend on any choices, and therefore
our discussion of such objects is very terse.

Essentially, a non-principal ultrafilter $\omega$ picks out a limit of every absolutely 
bounded sequence $a_n$,
which we denote $\lim_\omega a_n$. We can think of $\lim_\omega$ as a function from
$\ell^\infty$ to $\R$; this function is a ring homomorphism. The ``limit'' $\lim_\omega$ 
is always contained between $\limsup$ and $\liminf$, and is equal to the honest limit of some infinite
subsequence of the $a_n$ (so, for instance, if $\lim_{n\to \infty} a_n$ exists it is equal to
$\lim_\omega a_n$) but otherwise satisfies no {\it a priori} constraint. For more details,
see Gromov {\it op. cit.}

Fix (once and for all) a non-principal ultrafilter $\omega$.

\begin{definition}
Given a subset $W \subset F$ and a group $G$, define the asymptotic cone $\widehat{A}_W(G)$ (or just 
$\widehat{A}_W$ or $\widehat{A}$
if $G$ is understood) to be the ultralimit of the sequence of metric spaces $C_W(G)$ with the $W$-metric
rescaled by a factor of $1/n$. 

Thus a point in $\widehat{A}_W(G)$ is an equivalence class of sequence 
$\seq(a_n)$ with $d(\id,a_n)=O(n)$, where $d(\seq(a_n),\seq(b_n)) = \lim_\omega d(a_n,b_n)/n$,
and where $\seq(a_n)$ and $\seq(b_n)$ are in the same equivalence class
(denoted $\seq(a_n)\sim\seq(b_n)$) iff $\lim_\omega d(a_n,b_n)/n = 0$.
\end{definition}

\begin{proposition}\label{asymptotic_group}
The asymptotic cone $\widehat{A}_W$ has the structure of a group, with a bi-invariant metric, satisfying the
law $W$.
\end{proposition}
\begin{proof}
The group structure on $\widehat{A}$ is defined by taking $\seq(\id)$ (the constant equivalence class) 
as the identity, defining multiplication
by $\seq(a_n) \cdot\seq(b_n) = \seq(a_nb_n)$, and inverses by $\seq(a_n)^{-1} = \seq(a_n^{-1})$.
To see that this is well-defined, suppose $\seq(a_n)\sim\seq(a_n')$ and $\seq(b_n)\sim\seq(b_n')$.
By definition, we have
$a_n = w_na_n'$ and $b_n = v_nb_n'$ where $\l(w_n|W)=d(a_n,a_n')$ and $\l(v_n|W)=d(b_n,b_n')$.
Hence
$$a_nb_n = w_na_n'v_nb_n' = w_n v_n^{a_n'}a_n'b_n'$$
so 
$$\lim_\omega d(a_nb_n,a_n'b_n')/n \le \lim_\omega (d(a_n,a_n') + d(b_n,b_n'))/n = 0$$
The bi-invariance of the metric on $\widehat{A}$ follows immediately from the definition of the multiplication,
and the bi-invariance of the metric on $C$.

Finally, if $\seq(a_{1,n})\seq(a_{2,n}) \cdots \seq(a_{m,n})$ is a $W$-word in $\widehat{A}$ (with its
group structure), then $a_{1,n}a_{2,n} \cdots a_{m,n}$ is a $W$-word in $G$ for each $n$,
and therefore $$\seq(a_{1,n})\seq(a_{2,n}) \cdots \seq(a_{m,n})=\seq(a_{1,n}a_{2,n} \cdots a_{m,n})\sim\seq(\id)$$
\end{proof}

\subsection{The group \protect{$B_W$}}

We now relate $\widehat{A}_W$ to $G_W$ (in fact, to a slightly larger group depending on $G$ and $W$). 

\begin{lemma}\label{continuity_in_lines}
Let $W \subset F$ be given, and let $g_1,g_2,\cdots,g_m$ be in $G_W$. There is an estimate
$$\l((g_1^{a_1}g_2^{a_2}\cdots g_m^{a_m})^{-1}(g_1^{b_1}g_2^{b_2}\cdots g_m^{b_m})|W) \le \sum_i \l(g_i^{|b_i-a_i|}|W)$$
\end{lemma}
\begin{proof}
By expanding,
\begin{align*}
\text{LHS} &= g_m^{-a_m}g_{m-1}^{-a_{m-1}}\cdots g_2^{-a_2} g_1^{b_1-a_1} g_2^{b_2}\cdots g_m^{b_m} \\
& = g_m^{-a_m}g_{m-1}^{-a_{m-1}}\cdots g_2^{b_2-a_2}\cdots g_m^{b_m} (g_1^{b_1-a_1})^* \\
& = \prod_i (g_i^{b_i-a_i})^* 
\end{align*}
\end{proof}

\begin{remark}
It would be nice to generalize Lemma~\ref{continuity_in_lines} to certain
homogeneous expressions $g_1g_2\cdots g_m$ whose
{\em product} is in $G_W$, even if the individual $g_i$ are not. In fact, this can be done if $m=2$,
but {\em not} for higher $m$ unless $G_W = [G,G]$ (i.e.\ unless every commutator is a product of
$W$-words).

Suppose $gh \in G_W$. Then 
\begin{align*}
g^nh^n &= (gh)^n [g^{n-1},h]^*[g^{n-2},h]^* \cdots [g,h]^*\\
       &= (gh)^n [g^{n-1},gh]^* \cdots [g,gh]^*
\end{align*}
where we use the identity $[a,b] = [a,ab]^*$. Since any commutator $[*,gh]$ 
is a product of at most $2\cdot\l(gh|W)$ $W$-words, we see that $\l(g^nh^n|W) \le 3n\cdot\l(gh|W)$.
Moreover,
$$(g^{n_1}h^{n_1})^{-1}(g^{n_2}h^{n_2})= (g^{n_2-n_1}h^{n_2-n_1})^*$$
and we derive the inequality $$\l((g^{n_1}h^{n_1})^{-1}(g^{n_2}h^{n_2})|W) \le 3|n_2-n_1|\cdot\l(gh|W)$$
which is analogous to the estimate in Lemma~\ref{continuity_in_lines}.

On the other hand, in general the condition that $xyz$ is in $G_W$ does {\em not} imply that $x^ny^nz^n$
is in $G_W$ for all $n$ (in fact, for any $n$ other than $0,1$), unless $G_W=[G,G]$.
For example, if $y,z$ are arbitrary and $x$ is chosen so that $xyz \in G_W$, then
$$x^{-1}y^{-1}z^{-1} = x^{-1}z^{-1}y^{-1}[y,z]$$
Now, $x^{-1}z^{-1}y^{-1}$ is in $G_W$ because $xyz$ is, so the left hand side is in $G_W$ iff $[y,z]$
is. But $y$ and $z$ are arbitrary, so this holds for all $y,z$ if and only if $G_W=[G,G]$, as claimed.
\end{remark}

Let $B_W$ be the free group obeying the law $W$ generated by the elements of $G_W$ {\em as a set},
subject to the relation of homogeneity $g^n = g^{\star n}$ where $\star$ is the group operation in
$B_W$. There is a natural inclusion $G_W \to B_W$ of {\em sets}, taking each element of $G_W$ to
the corresponding generator.

\begin{lemma}
There is a homomorphism $\rho:B_W \to \widehat{A}_W$ defined by the formula
$$\rho(g_1\star g_2 \star \cdots \star g_m) = \seq(g_1^ng_2^n\cdots g_m^n)$$
\end{lemma}
\begin{proof}
From the definition of the group operation in $\widehat{A}_W$, the right hand side is equal to
the product $\seq(g_1^n)\seq(g_2^n)\cdots \seq(g_m^n)$, so the map is a homomorphism.
To see that it is well-defined, observe that all relations (i.e.\/ $W$-words, and the
homogeneity relations) are satisfied in $\widehat{A}_W$ by Proposition~\ref{asymptotic_group}, 
so the map is well-defined.
\end{proof}

\begin{remark}
It is not typically true that this homomorphism is an injection, even if $G$ is free.
\end{remark}

This homomorphism lets us map $G_W$ to $\widehat{A}_W$ by $g \to \seq(g^n)$. 
From the definitions, $\sl(g|W) = d(\seq(\id),\seq(g^n))$; in other words {\em stable $w$-length
can be recovered from the geometry of $\widehat{A}_W$ and the map $B_W \to \widehat{A}_W$}.

\subsection{Real structure}

For any group $G$, let $B_W \otimes \R$ be the free group obeying the law $W$ generated by expressions
of the form $g^t$ where $g \in G_W$ and $t \in \R$, subject to the relation of homogeneity $g^s\star g^t = g^{s+t}$.
We topologize $B_W \otimes \R$ with the weakest topology for which multiplication and inverse are
continuous, as well as each homomorphism $\R \to B_W \otimes \R$ of the form $t \to g^t$.

We extend $\rho:B_W \to \widehat{A}_W$ to $\rho:B_W \otimes \R$ by defining
$$g^t \to \seq(g^{\fl(tn)})$$
on the $\R$ subgroups, and extending to arbitrary products by using the group multiplication. Hence
$$g_1^{t_1}g_2^{t_2}\cdots g_m^{t_m} \to \seq(g_1^{\fl(t_1n)}g_2^{\fl(t_2n)}\cdots g_m^{\fl(t_mn)})$$

Notice that although $g^{\fl(tn)}$ is not necessarily the inverse of $g^{\fl(-tn)}$ if $tn$ is
not an integer, nevertheless these elements are distance at most $2\l(g|W)$ apart in $C_W$, and therefore
$\seq(g^{\fl(tn)})$ and $\seq(g^{\fl(-tn)})$ are inverse in $\widehat{A}_W$.

\begin{definition}
The {\em real cone} $A_W$ is the image of $B_W \otimes \R$ in $\widehat{A}_W$.
\end{definition}

For the remainder of the paper we restrict attention to $A_W$.
The metric on $\widehat{A}_W$ restricts to a (bi-invariant) metric on $A_W$,
consequently giving it the structure of a topological group. In the sequel, we use the notation
$\|g\|$ for $d(\id,g)$, where $g\in A_W$.

The real structure on $B_W\otimes \R$ gives rise to a natural (multiplicative) $\R$ action on $A_W$.

\begin{lemma}\label{R_action}
There is a continuous family of endomorphisms $\R \times A_W \to A_W$ with the following properties:
\begin{enumerate}
\item{the action is multiplicative --- i.e.\ $1\times$ is the identity on $A_W$, and
$\lambda \times (\mu\times g) = (\lambda\mu)\times g$ for $\lambda,\mu \in \R$ and $g \in A_W$; and}
\item{for any $g \in A_W$ and $\lambda \in \R^*$, we have $\|\lambda\times g\| = |\lambda|\cdot\|g\|$}
\end{enumerate}
\end{lemma}
\begin{proof}
Let $B$ be the group freely generated by expressions of the form $g^t$ with $g \in G$ and $t\in \R$,
subject to the relations $g^t g^s = g^{t+s}$. There is a natural $\R$ action on $G$, given by
$\lambda\times g^t = g^{\lambda t}$. This action evidently preserves the $W$-subgroup of $B$, and
therefore descends to an automorphism of $B_W\otimes \R$. We claim that this action preserves the
kernel of $\rho$, and thereby defines an action on $A_W$. In other words, we need to show that
if $\|\rho(g_1^{t_1}\cdots g_m^{t_m})\|=0$ then 
$\|\rho(g_1^{\lambda t_1}\cdots g_m^{\lambda t_m})\|=0$ for any $\lambda$. But
\begin{align*}
\|\rho(g_1^{\lambda t_1}\cdots g_m^{\lambda t_m})\| &= d(\seq(\id),\seq(g_1^{\fl(\lambda t_1 n)}\cdots g_m^{\fl(\lambda t_m n)} ) ) \\
&= \lim_\omega \frac 1 n \; \l(g_1^{\fl(\lambda t_1 n)}\cdots g_m^{\fl(\lambda t_m n)}|W) \\
&= \lambda \cdot \lim_\omega \frac 1 n \; \l(g_1^{\fl(t_1 n)}\cdots g_m^{\fl(t_m n)}|W) \\
&= \lambda \cdot \|\rho(g_1^{t_1}\cdots g_m^{t_m})\|
\end{align*}
(where the third line follows from the second by an approximate change of variables and elementary estimates)
thereby establishing both claims.
\end{proof}

The existence of this $\R$ action has the following topological consequence.

\begin{lemma}\label{locally_contractible}
$A_W$ is contractible and locally contractible.
\end{lemma}
\begin{proof}
The $\R$ action defines a deformation retraction of $A_W$ to the identity element, where each element
$g$ moves along the path $(1-t)\times g$ where $t$ goes from $0$ to $1$. This retraction takes the
ball of radius $r$ around the identity inside itself for all positive $t$, so $A_W$ is locally contractible.
\end{proof}

Since $\widehat{A}_W$
obeys the law $W$, so does $A_W$. However, under certain circumstances we can say much
more about $A_W$. The main theorem of this section is the following:

\begin{theorem}\label{nilpotent_implies_abelian}
Suppose $A_W$ is nilpotent. Then it is a normed vector space (in particular it is abelian).
\end{theorem}
\begin{proof}
If $A_W$ is abelian, then the existence of the $\R$ action with the properties proved in Lemma~\ref{R_action}
shows that $A_W$ is a normed vector space, with norm $\|\cdot\|$. So it suffices to show $A_W$ is abelian.

For legibility, denote the $n$th element of the lower central series of $A_W$ by $A_n$, so that $A_1 = A_W$ and
$A_n = 0$ for some $n$. Suppose we have shown for some integer $k$, for all real $t$, for all $h \in A_k$ and all $g\in G_W$, that
the commutator $[\rho(g^t),h]=\id$. Since $A_W$ is generated by elements of the form $\rho(g^t)$ this implies
that $A_{k+1}=0$.

Let $h \in A_{k-1}$. Then 
\begin{align*}
[\rho(g^t),h] & = [\rho(g^{t/2}),h][\rho(g^{t/2}),h][[h,\rho(g^{t/2})],\rho(g^{t/2})^h] & \text{ by Lemma~\ref{elementary_identities} }\\ 
& = [\rho(g^{t/2}),h][\rho(g^{t/2}),h] & \text{ because }A_{k+1}=0 \\
& = [\rho(g^{t/2}),h^2][[h,\rho(g^{t/2})],h]^{-1} & \text{ by Lemma~\ref{elementary_identities} } \\
& = [\rho(g^{t/2}),h^2] & \text{ because }A_{k+1}=0
\end{align*}
By induction, $[\rho(g^t),h] = [\rho(g^{t/2^m}),h^{2^m}]$ for every positive integer $m$. Hence
$$\|[\rho(g^t),h]\| = \|[\rho(g^{t/2^m}),h^{2^m}]\| = \|g^{t/2^m} \cdot (g^{-t/2^m})^*\| \le 2\cdot 2^{-m}\|g^t\|$$
Since $m$ is arbitrary, $[\rho(g^t),h]=\id$ and therefore (assuming $k\ge 2$),
$A_k=0$. Since $A_W$ is nilpotent by hypothesis,
$A_n=0$ for some $n$, and therefore $A_W$ is abelian, as claimed.
\end{proof}

As a corollary, one obtains the following propositions.

\begin{proposition}
Suppose $A_W$ is locally compact. Then it is a (finite dimensional) normed vector space.
\end{proposition}
\begin{proof}
Since $A_W$ is connected, locally path connected (by Lemma~\ref{locally_contractible}) and locally
compact by hypothesis, the Gleason-Montgomery-Zippin theorem (i.e.\ the affirmative solution of
Hilbert's fifth problem \cite{Gleason,Montgomery_Zippin}) shows that $A_W$ is a Lie group. 
Since it obeys a law, it is necessarily solvable, since non-solvable Lie groups contain 
nonabelian free subgroups.

By Theorem~\ref{nilpotent_implies_abelian}, it suffices to show it is nilpotent. Equivalently,
we need to show that for each $g\in G_W$ and $t \in \R$ the conjugation action of $\rho(g^t)$ on
the derived subgroup $A'$ is trivial. 

Suppose not, so that $\rho(g^t)$ is a $1$-parameter family of isometric automorphisms of $A'$. Since $A'$ is
finite-dimensional, the group of isometric automorphisms is compact. In particular, there are arbitrarily large
values of $t$ for which $\rho(g^t)$ is arbitrarily close to the identity in the isometry group
of $A'$. 

By the compactness of the group of isometric automorphisms of $A'$
there is a constant $C$ so that if $\iota$ is an automorphism with $d(h,\iota(h))<\epsilon$
for all $h$ in the ball about the identity in $A'$
of radius $1/2$, then $d(h,\iota(h))<C\cdot\epsilon$ for all $h$ in the ball about the
identity in $A'$ of radius $1$ (in fact, since the metric on $A'$ is bi-invariant, we can take
$C=2$, but this is unnecessary for our argument).
 
Now, for any positive $\epsilon$, there are arbitrarily large
$t$ such that $d(h,h^{\rho(g^t)})<\epsilon$
for all $h$ in the ball of radius $1$ in $A'$. Consequently, for all $s \in [1/2,1]$ and all
$h$ in the ball of radius $1$ in $A'$,
$$d(s\times h, s\times (h^{\rho(g^t)}))  = d(s\times h,(s\times h)^{\rho(g^{ts})}) < s\epsilon$$
It follows that $d(h,h^{\rho(g^{ts})})<s\epsilon C$ for all $h$ in the ball of radius $1$ in $A'$,
and all $s \in [1/2,1]$. It follows that $d(h,h^{\rho(g^{t'})})<2\epsilon C$ for all $t' \in [0,t/2]$.

Since $\epsilon$ is arbitrarily small, since $C$ is fixed, and since $t$ can be chosen arbitrarily
large for each $\epsilon$, it follows that the conjugation action of $\rho(g^t)$ is trivial for
all $g \in G_W$ and $t \in \R$. Hence $A'$ is central, and therefore $A_W$ is nilpotent, and therefore
abelian and a finite dimensional normed vector space, as claimed.
\end{proof}

\begin{remark}
In fact, we do not really need the full power of Gleason-Montgomery-Zippin. We only need to know
that the group of isometries of $A'$ is compact (which follows from Arzela-Ascoli and local compactness)
to deduce that $A$ is abelian; then one may appeal to Pontriagin's solution of Hilbert's fifth problem 
for locally compact abelian groups.
\end{remark}

Recall the notation $\beta_2:=[[x,y],[z,w]]$ for free generators $x,y,z,w$.

\begin{theorem}\label{vanishing_scl_beta_2_abelian}
Suppose in some group $G$ that stable commutator length vanishes identically. 
Then $A_W$ is abelian (and hence a normed vector space) for $W=\beta_2$. 
\end{theorem}
\begin{proof}
By Theorem~\ref{nilpotent_implies_abelian} it suffices to show that $A_W$ is nilpotent. We show it
satisfies the law $[x,[y,z]]=0$.

Let $f,g,h$ be arbitrary elements of $A_W$, and suppose $f=\rho(f_1^{t_1}f_2^{t_2}\cdots f_m^{t_m})$. We
write $g=\seq(g_n)$, $h=\seq(h_n)$ and $f=\seq(f_1^{\fl(nt_1)}f_2^{\fl(nt_2)}\cdots f_m^{\fl(nt_m)})$.
By hypothesis, the commutator length of each expression $f_i^{\fl(nt_i)}$ is $o(n)$. It follows that
$$[f,[g,h]] = \seq([f_1^{\fl(nt_1)}f_2^{\fl(nt_2)}\cdots f_m^{\fl(nt_m)},[g_n,h_n]])=\seq({o(n) \; \beta_2-\text{words}})=0$$
where we use Lemma~\ref{elementary_identities} bullet (4) in the second step. 
Hence $A_W$ is nilpotent and therefore abelian.
\end{proof}

\begin{remark}
Recall that $\beta_1=[x,y]$, so that $\beta_1$-length is just commutator length. A similar argument shows
that if stable $\beta_n$-length vanishes identically in $G$, then for $W=\beta_{n+1}$ the group $A_W$
obeys the law $[x,\beta_n]$.
\end{remark}

\subsection{$\gamma_n$-quasimorphisms}

For the remainder of this section we typically
specialize to the case that $w=\gamma_n$ (recall $\gamma_2=[x,y]$,
$\gamma_3=[x,[y,z]]$ and so on). Inspired by the phenomenon of (generalized) {\em Bavard duality}
in the theory of stable commutator length, we make the following definition:

\begin{definition}
A {\em homogeneous $W$-quasimorphism} (hereafter {\em $W$-hoq}) is a Lipschitz homomorphism from
$A_W$ to $\R$.
\end{definition}

If $\phi$ is a $W$-hoq, let $d(\phi)$ denote the optimal Lipschitz constant.

For general $W$, it is not clear whether any homomorphisms from $A_W$ to $\R$ exist, let alone
Lipschitz ones. But when $A_W$ is a normed vector space (which holds, for instance, if $W=\gamma_n$,
by Theorem~\ref{nilpotent_implies_abelian}) the Hahn-Banach theorem guarantees the existence of
a rich supply of $W$-hoqs.

In fact, we have the following proposition:

\begin{lemma}\label{w_length_from_hoqs}
For any $n$ and any $g \in G_n$, there is an equality
$$\sl(g|\gamma_n) = \sup_{\phi} \phi(\rho(g))/d(\phi)$$
where the supremum is taken over all $\gamma_n$-hoqs $\phi$.
\end{lemma}
\begin{proof}
This follows from the fact that $\|\rho(g)\| = \sl(g|\gamma_n)$, together with
Theorem~\ref{nilpotent_implies_abelian} and the Hahn-Banach theorem.
\end{proof}

Whenever $A_W$ is a vector space, it is spanned by the vectors $\rho(g)$ for $g \in G_W$. Since
a $W$-hoq $\phi$ is a homomorphism, such a function on $A_W$ is determined by its values on
$\rho(g)$. Therefore, by abuse of notation, we think of a $W$-hoq in this case as a function on
$G_W$; thus the equality from proposition~\ref{w_length_from_hoqs} would be expressed in the form
$\sl(g|\gamma_n) = \sup_{\phi} \phi(g)/d(\phi)$. On the other hand, it is not evident from
the definition how to recover (or to estimate)
$d(\phi)$ directly from the values of $\phi$ on the elements of $G_W$. We now address this issue.

\begin{definition}\label{weak_hoq_definition}
A {\em weak} $\gamma_n$-hoq is a function $\phi:G \to \R$ for which there is a least non-negative
real number $D(\phi)$ (called the {\em defect}) satisfying the following properties:
\begin{enumerate}
\item{(homogeneity) for any $g\in G$ and any $n \in \Z$ there is an equality $\phi(g^n)=n\phi(g)$;}
\item{(quasimorphism) for any $g,h \in G$ there is an inequality 
$$|\phi(g) + \phi(h) - \phi(gh)| \le D(\phi)\min(\l(g|\gamma_{n-1}),\l(h|\gamma_{n-1}),\l(gh|\gamma_{n-1}))$$}
\end{enumerate}
\end{definition}

\begin{remark}
For $n=2$, this is precisely the classical definition of a homogeneous quasimorphism. For, $\gamma_1=x$
and therefore $\l(g|\gamma_1)=1$ if $g\ne \id$ and $0$ otherwise.
\end{remark}

From the definition one can deduce some basic properties of weak $\gamma_n$-hoqs.

\begin{lemma}\label{weak_hoq_estimate}
Suppose $G$ is perfect. Then any weak $\gamma_n$-hoq $\phi:G\to \R$ satisfies the following properties:
\begin{enumerate}
\item{$\phi$ is a class function;}
\item{if $h$ is a $\gamma_n$-word, then $\phi(h)\le D(\phi)$;}
\item{for any $g\in G$ there is an estimate
$$\phi(g) \le (2\l(g|\gamma_n)-1)D(\phi)$$ and consequently
$$\sl(g|\gamma_n) \ge \phi(g)/2D(\phi)$$}
\end{enumerate}
\end{lemma}
\begin{proof}
The hypothesis that $G$ is perfect implies (inductively) that $\l(g|\gamma_m)$ is finite 
for every $g\in G$ and every integer $m$. To see that $\phi$ is a class function, observe 
that for any elements $g,h\in G$ and any $n$ one has
$$|\phi(hg^nh^{-1}) - \phi(h) - \phi(g^n) - \phi(h^{-1})| \le 2D(\phi)\l(g|\gamma_{n-1})$$
By homogeneity, the left hand side is equal to $n\cdot|\phi(hgh^{-1})-\phi(g)|$
whereas the right hand side is a constant independent of $n$, thus proving the first claim.

Now suppose $h$ is a $\gamma_n$-word; i.e.\ $h=[x,y]$ for some $x$ and $y$ 
where $y$ is a $\gamma_{n-1}$-word.
Then
$$|\phi(h) - \phi(xyx^{-1}) - \phi(y^{-1})| \le D(\phi)\l(y|\gamma_{n-1}) = D(\phi)$$
On the other hand, since $\phi$ is a class function and homogeneous, the left hand side is equal to
$|\phi(h)|$ and the inequality is proved.

Finally, if we express $g=h_1h_n\cdots h_m$ where each $h_i$ is a $\gamma_n$-word, then
by induction $\phi(g) \le \sum \phi(h_i) + (m-1)D(\phi) \le (2m-1)D(\phi)$. This proves bullet~(3).
\end{proof}

\begin{proposition}
The set of weak $\gamma_n$-hoqs on $G$ is a real vector space, which we denote $_nQ(G)$, and $D(\cdot)$ is
a semi-norm. If $G$ is perfect, $_nQ(G)$ is a Banach space with the norm $D(\cdot)$.
\end{proposition}
\begin{proof}
The defining properties of a weak $\gamma_n$-hoq are an (infinite) system of linear equalities and inequalities.
It follows by inspection that $_nQ(G)$ is a real vector space, and that $D(\cdot)$ is a semi-norm.

If $G$ is perfect, then $\l(g|\gamma_{n-1})$ is finite for any $g$, and therefore $D(\phi)=0$ if and
only if $\phi$ is a homomorphism to $\R$, which is necessarily trivial for $G$ perfect. Thus
$D(\cdot)$ is a norm, and it remains to show that $_nQ(G)$ is complete in this norm.

So, let $\phi_i$ be a sequence of weak $\gamma_n$-hoqs for which $D(\phi_i-\phi_j)$ is a Cauchy sequence; i.e.\
for all $\epsilon>0$ there is an $N$ so that $D(\phi_i-\phi_j)<\epsilon$ for $i,j>N$. By
bullet~(3) from Lemma~\ref{weak_hoq_estimate} it follows that the values of the $\phi_i$ on any
element $g$ form a Cauchy sequence, and therefore the $\phi_i$ converge pointwise to a limit $\phi$.
A pointwise limit of homogeneous functions is homogeneous; moreover, for any $g,h$ pointwise
convergence implies
$$|(\phi-\phi_j)(g+h-gh)| \le \limsup_i D(\phi_i-\phi_j)\min(\l(g|\gamma_{n-1}),\l(h|\gamma_{n-1}),\l(gh|\gamma_{n-1}))$$
and therefore $\phi$ is in $_nQ(G)$ and $D(\phi_i-\phi)\to 0$, and the lemma is thereby proved.
\end{proof}

The following lemma justifies the terminology ``weak $\gamma_n$-hoq''.

\begin{lemma}\label{hoq_is_weak_hoq}
If $G$ is perfect, every $\gamma_n$-hoq is a weak $\gamma_n$-hoq with $D(\phi)\le d(\phi)$.
\end{lemma}
\begin{proof}
Let $\phi:A_{\gamma_n} \to \R$ be a $d(\phi)$-Lipschitz homomorphism, and by abuse
of notation we let $\phi:G \to \R$ be the function defined by $\phi(g):=\phi(\rho(g))$.
We estimate 
\begin{align*}
|\phi(g) + \phi(h) - \phi(gh)|&= |\phi(\rho(g*h*(h^{-1}g^{-1})))| \\
&\le d(\phi)\|\seq(g^mh^m h^{-1}g^{-1}\cdots h^{-1} g^{-1})\| \\
&=d(\phi) \lim_{m \to \infty} \frac 1 m \l(g^mh^m h^{-1}g^{-1}\cdots h^{-1} g^{-1}|\gamma_n) \\
&\le d(\phi) \lim_{m \to \infty} \frac 1 m \l([g,h]^*[g^2,h]^*\cdots [g^{m-1},h]^*|\gamma_n) \\
&\le d(\phi) \l(h|\gamma_{n-1})
\end{align*}
By symmetry, we obtain the estimate
$$|\phi(g) + \phi(h) - \phi(gh)| \le d(\phi)\min(\l(g|\gamma_{n-1}),\l(h|\gamma_{n-1}),\l(gh|\gamma_{n-1}))$$
so that $\phi$ is a weak $\gamma_n$-hoq with $D(\phi) \le d(\phi)$, as claimed.
\end{proof}

We therefore obtain the following generalization of Bavard duality.

\begin{theorem}[$\gamma_n$-Duality theorem]\label{duality_theorem}
For any perfect group $G$, and any $g \in G$ there is an inequality
$$\sup_\phi \phi(g)/D(\phi) \ge \sl(g|\gamma_n) \ge \sup_\phi \phi(g)/2D(\phi)$$
where the supremum is taken over all weak $\gamma_n$-hoqs.
\end{theorem}
\begin{proof}
The upper bound follows from Lemma~\ref{hoq_is_weak_hoq} and Lemma~\ref{w_length_from_hoqs}, 
and the lower bound from  Lemma~\ref{weak_hoq_estimate}.
\end{proof}

An interesting application of this duality theorem is an {\it a priori} estimate of the
ratio of $\scl(*)$ and $\sl(*|\gamma_n)$ in any perfect group. This estimate follows
from the surprising fact that one may bound from below the stable $\gamma_{n-1}$ length of a word from
the value of a weak $\gamma_n$-hoq, providing $n\ge 3$.

\begin{lemma}\label{one_better_lemma}
Let $G$ be a perfect group. Then any weak $\gamma_n$-hoq $\phi:G \to \R$ satisfies the
following properties:
\begin{enumerate}
\item{if $h$ is a $\gamma_{n-1}$-word, then $\phi(h)\le D(\phi)$}
\item{for any $g \in G$ there is an estimate
$$\phi(g) \le (2\l(g|\gamma_{n-1})-1)D(\phi)$$
and consequently
$$\sl(g|\gamma_{n-1}) \ge \phi(g)/2D(\phi)$$}
\end{enumerate}
\end{lemma}
\begin{proof}
Let $g$ be arbitrary, and let $k$ be a $\gamma_{n-2}$-word. There is a constant $C$ so that
for any $m$ we have $\l(g^mk^m(gk)^{-m}|\gamma_n) \le Cm$
and therefore by Lemma~\ref{weak_hoq_estimate}
we can estimate
$$|\phi(g^mk^m(gk)^{-m})| \le C'm$$
for some constant $C'$ depending only on $g$, $k$ and $D(\phi)$.

On the other hand, we can write 
\begin{align*}
g^mk^m(gk)^{-m} &= g^m k^m k^{-1} g^{-1} k^{-1} g^{-1} \cdots k^{-1} g^{-1} \\
&= g^m k^{m-1} g^{-1} k^{-1} \cdots k^{-1} g^{-1} \\
&= g^m k^{m-2} g^{-1} [g,k] g^{-1} k^{-1} \cdots k^{-1} g^{-1} \\
&= g^m k^{m-2} g^{-1} g^{-1} k^{-1} \cdots k^{-1} g^{-1} [g,k]^* \\
&= g^m k^{m-3} g^{-1} g^{-1} g^{-1} k^{-1} \cdots k^{-1} g^{-1} [g,k]^* [g,k]^* [g,k]^* \\
&= \text{product of } m(m-1)/2 \text{ conjugates of } [g,k]
\end{align*}
By abuse of notation, for any $r$ we denote a product of 
$r$ conjugates of $[g,h]$ by $[g,h]^{*r}$. Hence we have $g^mk^m(gk)^{-m} = [g,k]^{*m(m-1)/2}$. On the other hand,
for any $r$,
$$|\phi([g,k]^{*r}) - \phi([g,k]^{*r-1}) - \phi([g,k])| \le D(\phi)$$
by the defining property of an $\gamma_n$-hoq, since $\phi$ is a class function, and 
$\l(k|\gamma_{n-2})=1$ implies $\l([g,k]|\gamma_{n-1})\le 1$. By induction and the
triangle inequality we obtain
$$|\phi([g,k]^{*r}) - r\phi([g,k])| \le r D(\phi)$$
Hence we conclude
$$|\phi(g^mk^m(gk)^{-m}) - (m(m-1)/2) \phi([g,k])| \le (m(m-1)/2) D(\phi)$$
Since $|\phi(g^mk^m(gk)^{-m}|\le C'm$, dividing by $m(m-1)/2$ and taking $m \to \infty$ gives the
estimate
$$|\phi([g,k])| \le D(\phi)$$
Since $g$ was arbitrary, and $k$ is an arbitrary $\gamma_{n-2}$-word, we have proved bullet (1).
Bullet (2) follows exactly as in Lemma~\ref{weak_hoq_estimate}: if $g=h_1h_2\cdots h_m$ where
each $h_i$ is a $\gamma_{n-1}$-word, then by induction
$$\phi(g) \le \sum \phi(h_i) + (m-1)D(\phi) \le (2m-1)D(\phi)$$
This completes the proof of the lemma.
\end{proof}

From this easily follows our Comparison Theorem:

\begin{theorem}[Comparison theorem]\label{comparison_theorem}
For any perfect group $G$, for any $n\ge 2$ and for any $g\in G$ there is an inequality
$$2^{n-2}\scl(g) \ge \sl(g|\gamma_n) \ge \scl(g)$$
\end{theorem}
\begin{proof}
From Lemma~\ref{one_better_lemma} and Theorem~\ref{duality_theorem} we estimate
$$2\sl(g|\gamma_{n-1}) \ge \sup_\phi \phi(g)/D(\phi) \ge \sl(g|\gamma_n)$$
where the supremum is taken over all weak $\gamma_n$-hoqs $\phi$. By induction on $n$, we get
$$2^{n-2}\scl(g) \ge \sl(g|\gamma_n)$$
The second inequality in the theorem is trivial, since every $\gamma_n$-word is a $\gamma_2$-word.
\end{proof}

It follows that weak $\gamma_n$-hoqs can be used to estimate $\scl$:

\begin{corollary}\label{scl_bound_from_weak_n_hoq}
Let $G$ be a perfect group, and $\phi$ a weak $\gamma_n$-hoq. Then for any $g\in G$ we have
$$\scl(g) \ge \phi(g)/2^{n-1}D(\phi)$$
\end{corollary}

Stable commutator length has many important applications to geometry, topology, dynamics, etc. In
general, estimating $\scl$ is very difficult; lower bounds are usually obtained from (ordinary homogeneous)
quasimorphisms (equivalently, weak $\gamma_n$-hoqs when $n=2$), and by now many interesting constructions
of such quasimorphisms are known; see e.g.\ \cite{Calegari_scl}, especially Chapters 3 and 5.
In view of Corollary~\ref{scl_bound_from_weak_n_hoq}, it becomes interesting to ask whether there
are any natural constructions of weak $\gamma_n$-hoqs with $n>2$, arising from geometry or dynamics.

\medskip

The statement of the Comparison Theorem is purely algebraic, and it is therefore natural to try to
find a purely algebraic proof, bypassing the construction of asymptotic cones, the use of
Hahn-Banach, etc. In fact, it is not too hard to translate the geometric argument into an algebraic
one, though the geometric argument has its own charm. The case of $\gamma_3$ is especially straightforward;
with the same amount of work, one proves a slightly stronger statement.

First, in any group $G$, let $\Gamma_3$ denote the set of words of the form $[x,y]$ where $y\in [G,G]$.
The $\Gamma_3$ words generate the subgroup $G_3$ (as do the $\gamma_3$ words, which are special examples of
$\Gamma_3$ words), but there is no uniform comparison between $\gamma_3$ length and $\Gamma_3$ length.
If $G$ is perfect, then $G=[G,G]=G_2$, so for such groups, $\Gamma_3$ words are nothing other than commutators,
and $\l(*|\Gamma_3) = \cl(*)$.

\begin{proposition}\label{algebraic_gamma_3}
In any group $G$, for any $g\in G_3$ there is an inequality
$$2 \sl(g|\Gamma_3) \ge \sl(g|\gamma_3) \ge \sl(g|\Gamma_3)$$
\end{proposition}
\begin{proof}
Recall (from Notation~\ref{seq_notation}) the notation 
$\gamma_3^{\seq(k)}$ for an arbitrary product of $\gamma_3$-words.

Express $g^k$ as a product of commutators
$$g^k = [a_1,b_1][a_2,b_2]\cdots[a_m,b_m]$$
where the $b_i$ are all in $[G,G]$, and
where $m/k$ is as close to $\sl(g|\Gamma_3)$ as we like. Then we have
$$g^{2^nk} = ([a_1,b_1]\cdots[a_m,b_m])^{2^n} = [a_1,b_1]^{2^n}\cdots[a_m,b_m]^{2^n} \gamma_3^\seq(m2^n)$$ 
Now, for each $i$ we have $[a_i,b_i]^2 = [a_i^2,b_i] \gamma_3^\seq(1)$
and therefore
$$[a_i,b_i]^{2^n} = [a_i^2,b_i]^{2^{n-1}}\gamma_3^\seq(2^{n-1}) = [a_i^{2^n},b_i]\gamma_3^\seq(2^n-1)$$
and therefore we can write
$$g^{2^nk} = [a_1^{2^n},b_1][a_2^{2^n},b_2]\cdots[a_m^{2^n},b_m]\gamma_3^\seq(m2^n-m)$$
We can estimate $\l([a_1^{2^n},b_1][a_2^{2^n},b_2]\cdots[a_m^{2^n},b_m]|\gamma_3) \le \sum_i \l(b_i|\gamma_2)$,
which is a constant independent of $n$. Since $m/k$ is as close as we like to $\scl(g)$, 
and $n$ is arbitrary, we deduce $\sl(g|\gamma_3) \le 2\scl(g)$ as claimed.
\end{proof}

We do not know whether there is an {\it a priori} comparison between $\sl(*|\gamma_3)$ and $\scl(*)$ on
elements of $G_3$ for an arbitrary group. 

\subsection{$\beta_2$-quasimorphisms}

By Theorem~\ref{vanishing_scl_beta_2_abelian}, if $G$ is a perfect group in which $\scl$ vanishes identically,
then $A_{\beta_2}$ is a vector space, and $\sl(g|\beta_2) = \sup_\phi \phi(g)/d(\phi)$ where the
supremum is taken over all $\beta_2$-hoqs $\phi$. We would like to obtain a ``weak'' characterization
of $\beta_2$-hoqs on such groups, analogous to the definition of weak $\gamma_n$-hoqs.

\begin{definition}
A {\em weak} $\beta_2$-hoq is a function $\phi:G \to \R$ for which there is a least non-negative
real number $D(\phi)$ (called the {\em defect}) satisfying the following properties:
\begin{enumerate}
\item{(homogeneity) for any $g\in G$ and any $n \in \Z$ there is an equality $\phi(g^n) = n\phi(g)$}
\item{(quasimorphism) for any $g,h \in G$ there is an inequality
$$|\phi(g) + \phi(h) - \phi(gh)| \le D(\phi)\min(\l(g|\beta_2),\l(h|\beta_2),\l(gh|\beta_2))$$}
\end{enumerate}
\end{definition}

The following is the analogue of Lemma~\ref{hoq_is_weak_hoq}:
\begin{lemma}\label{beta_2_hoq_is_weak_hoq}
If $G$ is perfect and $\scl$ vanishes identically, every $\beta_2$-hoq is a weak $\beta_2$-hoq 
with $D(\phi)\le 2d(\phi)$.
\end{lemma}
\begin{proof}
As in Lemma~\ref{hoq_is_weak_hoq}, we must estimate the $\beta_2$-length of $g^mh^m(gh)^{-m}$.
As before, we have an identity
$$g^mh^m(gh)^{-m} = [g,h]^*[g^2,h]^*\cdots[g^{m-1},h]^*$$
We can write $h=[a_1,b_1]\cdots[a_r,b_r]$ where each $a_i,b_i$ is a commutator, and
$r=\l(h|\beta_2)$. Then
$$[g^j,h] = [g^j,[a_1,b_1]]^*[g^j,[a_2,b_2]]^*\cdots[g^j,[a_r,b_r]]^*$$
By the Hall-Witt identity (i.e.\ bullet (5) from Lemma~\ref{elementary_identities}) we can
write
$$[g^j,[a_i,b_i]] = [a_i^*,[b_i^*,(g^j)^*]][b_i^*,[a_i^*,(g^j)^*]]$$
and since $a_i,b_i$ are both commutators by hypothesis, the right hand side is a product of
two $\beta_2$-words. Hence
$$\l(g^mh^m(gh)^{-m}|\beta_2) \le 2m\l(h|\beta_2)$$
By symmetry, the lemma follows.
\end{proof}

The following is the analogue of Lemma~\ref{weak_hoq_estimate}, though the proof is
more circuitous, and has something in common with that of Lemma~\ref{one_better_lemma}: 
\begin{lemma}\label{beta_2_weak_hoq_estimate}
Suppose $G$ is perfect, and $\scl$ vanishes identically in $G$. 
Then any weak $\beta_2$-hoq $\phi:G\to \R$ satisfies the following properties:
\begin{enumerate}
\item{$\phi$ is a class function;}
\item{if $h$ is a $\beta_2$-word, then $\phi(h)\le D(\phi)$;}
\item{for any $g\in G$ there is an estimate
$$\phi(g) \le (2\l(g|\beta_2)-1)D(\phi)$$ and consequently
$$\sl(g|\beta_2) \ge \phi(g)/2D(\phi)$$}
\end{enumerate}
\end{lemma}
\begin{proof}
To see that $\phi$ is a class function, observe as in the proof of Lemma~\ref{weak_hoq_estimate}
that for any $g,h$ and any $n$,
$$|\phi(hg^nh^{-1}) - \phi(h) - \phi(g^n) - \phi(h^{-1})| \le 2D(\phi)\l(h|\beta_2)$$
By homogeneity, and the fact that the right hand side is constant independent of $n$, we see
that $\phi(hgh^{-1})=\phi(g)$; i.e.\ $\phi$ is a class function.

Secondly, observe that if $\l(g|\beta_2)=1$ then $\phi([g,h])\le D(\phi)$. This
is because
$$|\phi(ghg^{-1}h^{-1}) - \phi(g) - \phi(hg^{-1}h^{-1})| \le D(\phi)\l(g|\beta_2) = D(\phi)$$
and because $\phi$ is a homogeneous class function. If we let $\alpha$ denote the word
$[[u,v],[[x,y],[z,w]]]$ then $\alpha$ is both a $\beta_2$-word, and the commutator of something
with a $\beta_2$ word. Thus by induction,
$|\phi(g)| \le 2D(\phi)\l(g|\alpha)$. Since $G$ is perfect, every element has a
finite $\alpha$-length, so for any $g$ and $h$
and any integer $m$ we have $|\phi(g^mh^m(gh)^{-m})| \le Cm$ for some constant $C$ depending
only on $g,h$ and $D(\phi)$.

On the other hand, as in the proof of Lemma~\ref{one_better_lemma}, we can write
$g^mh^m(gh)^{-m}$ a a product of $m(m-1)/2$ conjugates of $[g,h]$. If $g$ and $h$ are
both commutators, $[g,h]$ is a $\beta_2$-word, and therefore
$$|\phi(g^mh^m(gh)^{-m}) - (m(m-1)/2) \phi([g,h])| \le (m(m-1)/2) D(\phi)$$
and consequently $|\phi([g,h])| \le D(\phi)$. Since $g$ and $h$ are arbitrary commutators,
this shows that $|\phi(g)|\le D(\phi)$ for any $\beta_2$-word $g$, proving the second
claim.

The third claim follows immediately from this, as in the proof of Lemma~\ref{weak_hoq_estimate}. 
\end{proof}

\begin{theorem}[$\beta_2$-Duality theorem]\label{beta_2_duality_theorem}
For any perfect group $G$ in which $\scl$ vanishes identically, and for
any $g \in G$ there is an inequality
$$\sup_\phi 2\phi(g)/D(\phi) \ge \sl(g|\beta_2) \ge \sup_\phi \phi(g)/2D(\phi)$$
where the supremum is taken over all weak $\beta_2$-hoqs.
\end{theorem}
\begin{proof}
The upper bound follows from Lemma~\ref{beta_2_hoq_is_weak_hoq}, 
and the lower bound from  Lemma~\ref{beta_2_weak_hoq_estimate}.
\end{proof}

We conclude this section by making a curious observation on the
relation between $\sl(*|\beta_2)$ and $\l(*|\gamma_3)$, under the hypothesis that
$\scl$ vanishes identically.

\begin{proposition}\label{oddball_prop}
Suppose $G$ is perfect, and $\scl$ vanishes identically. If $\l(g|\gamma_3)=1$ then
$\sl(g|\beta_2)\le 1$.
\end{proposition}
\begin{proof}
Since for any $a,b,c$ we have $[a,b][a,c]=[a,bc][[a,c],b]^*$ by bullet (3) of
Lemma~\ref{elementary_identities}, it follows that
$$[x,[y,z]][x,[y,z]^n] = [x,[y,z]^{n+1}][[x,[y,z]^n],[y,z]]^*$$
for any $n$ and any $x,y,z$. Since $[[x,[y,z]^n],[y,z]]$ is a $\beta_2$-word, it follows by induction
that $[x,[y,z]]^n$ can be written as a product of $(n-1)$ $\beta_2$-words with $[x,[y,z]^n]$.
If $\scl$ vanishes identically, $[y,z]^n$ can be written as a product of $o(n)$ commutators.
If $G$ is perfect, $x$ is a product of a finite number of commutators. Hence the $\beta_2$-length
of $[x,[y,z]^n]$ is $o(n)$, and therefore $\sl([x,[y,z]]|\beta_2) \le 1$ as claimed.
\end{proof}

It seems hard to generalize Proposition~\ref{oddball_prop} to estimate $\sl(*|\beta_2)$ from
$\l(*|\gamma_3)$.

\subsection{Perfectness and virtual perfectness}

Throughout this section we have usually made the assumption that $G=G_W$. When $W=\gamma_n$
for some $n$, this is equivalent to the statement that $G$ is perfect. However it is evident 
that in most arguments it is sufficient to replace any given element $g$ with a (fixed) power
$g^m$. In particular, the theorems in this section remain true under the weaker hypothesis
that for every $g\in G$ there is a positive integer $m$ such that $g^m \in G_W$; equivalently,
the quotient $G/G_W$ is torsion.
For general $W$ this is implied by, but weaker than, the condition that $G_W$ has finite index in $G$,
even if $G$ is finitely generated. But for $W=\gamma_n$, the two conditions are equivalent when
$G$ is finitely generated.

\begin{lemma}
For any $n$, the quotient $G/G_n$ is torsion if and only if $G/G_2$ is torsion.
\end{lemma}
\begin{proof}
One direction is obvious. We prove the other direction. Suppose $G/G_2$ is torsion, and
suppose by induction we have shown $G/G_k$ is torsion for some $k\ge 2$. Let $g \in G_k$,
and write $g=[a_1,b_1][a_2,b_2]\cdots[a_r,b_r]$ where each $b_r \in G_{k-1}$. Let $m$ be
such that $b_i^m \in G_k$ for all $i$. Then 
\begin{align*}
g^m &= [a_1,b_1]^m\cdots[a_r,b_r]^m \mod G_{k+1} \\
&= [a_1,b_1^m]\cdots[a_r,b_r^m]\mod G_{k+1} \\
&= 0 \mod G_{k+1}
\end{align*}
By induction, the lemma is proved.
\end{proof}

Thus, Theorem~\ref{comparison_theorem} remains true with ``perfect'' replaced by ``virtually perfect''
(or even by: ``group whose abelianization is torsion'').

Given a group $G$, one may attempt to obtain a lower bound on $\sl(*|\gamma_n)$ on elements
of $G_n$ by embedding $G$ in a perfect group $H$, and using Theorem~\ref{duality_theorem}
applied to $H$ together with monotonicity of $\sl(*|*)$ under homomorphisms. It is
sometimes easier to take $H$ to be virtually perfect rather than perfect; we shall see
an example in \S~\ref{grope_section}.

\section{Gropes}\label{grope_section}

The purpose of this section is to show how hyperbolic geometry can be used to give straightforward
lower bounds on $\l(*|\gamma_3)$. We give some examples that show that the uniform comparisons 
in Theorem~\ref{comparison_theorem} and Proposition~\ref{algebraic_gamma_3} for {\em stable} 
lengths do not have any analogue for {\em unstable} length.

The arguments depend on the geometry of certain objects called {\em gropes}. We do not discuss here the
most general kind of gropes, but only the simplest nontrivial examples. For a general introduction
to gropes, see \S~13 of \cite{Cannon}. Informally speaking, gropes topologize the commutator calculus,
and questions about expressing elements in groups as products of $\gamma_n$ (or $\beta_n$) words
can be translated into questions about the existences of maps of certain kinds of gropes to spaces.
We also use some elementary facts from the theories of $\CAT(-1)$ complexes and pleated surfaces.
A basic reference for the first is \cite{Bridson_Haefliger}, especially pp.~347--362. A basic
reference for the second in \cite{Thurston}, Chapter~8.

\medskip

Let $S_n$ be an oriented surface of genus $n$ with one boundary component. Let $S_{n,1}$ be obtained by
attaching a once-punctured torus to each of a maximal collection of
pairwise disjoint homologically essential loops $\beta_i$ in $S_n$ (there are $n$ such). We denote
by $\partial S_{n,1}$ the boundary of $S_n$ (contained in $S_{n,1}$).

Now, let $G$ be a group, and let $X$ be a space with $\pi_1(X)=G$. If $g\in G$ is given,
let $\gamma:S^1 \to X$ be a loop whose free homotopy class corresponds to the conjugacy class of $g$.
Observe that from the definitions, $\l(g|\gamma_3) \le n$ if and only if there is a map
$f:S_{n,1} \to X$ such that $\partial f: \partial S_{n,1} \to X$ factors through a homeomorphism
$h: \partial S_{n,1} \to S^1$ in such a way that $\gamma\circ h = \partial f$. Informally, the
$\gamma_3$-length of $g$ is at most $n$ if and only if there is a map from $S_{n,1}$ to $X$ whose
boundary wraps around $\gamma$.

\begin{proposition}
For each $n$, let $F$ be free on the generators $x_1,y_1,\cdots,x_n,y_n,z$ and let
$w_n = [z,[x_1,y_1][x_2,y_2]\cdots[x_n,y_n]]$. Then $\l(w_n|\gamma_3)\ge 2n/3$.
\end{proposition}
\begin{proof}
We build a $K(F,1)$ (called $X$) as follows. Start with a hyperbolic once-punctured
torus $S$ with totally geodesic boundary. Let $\alpha$ be an embedded geodesic in $S$ representing
the meridian. Take another hyperbolic surface $S'$ of genus $n$ with totally geodesic boundary
of length $\length(\alpha)$, and attach $\partial S'$ isometrically to $\alpha$. The resulting space
$X$ is a $K(F,1)$, and the conjugacy class of $w_n$ is represented by the boundary circle
$\partial S$, which by abuse of notation, we denote $\partial X$.

Since $X$ is obtained by gluing convex hyperbolic $2$-complexes along convex subsets, it is itself a $\CAT(-1)$ 
$2$-complex. If $\l(w_n|\gamma_3)=m$,
there is a map $f:S_{m,1} \to X$ sending $\partial S_{m,1}$ to $\partial X$. We homotop this map to
a {\em pleated representative} in a special way. First we choose an ideal triangulation of $S_m$ 
for which the geodesic representatives of the 
$\beta_i$ (with notation as above) are contained in the pleating locus $\lambda$. The map $f$ can be homotoped on $S_m$ to take
each leaf of $\lambda$ to a geodesic in $X$, and to be $1$-Lipschitz on each ideal triangle of $S_m-\lambda$ for
some hyperbolic metric on $S_m$. Then $f$ can be homotoped rel. $\cup_i \beta_i$
to a pleated representative with respect to some hyperbolic metric on 
each once-punctured torus component of $S_{m,1} - S_m$.

The key property of a pleated map is that it is area non-increasing. Moreover, it is surjective onto $X$, since
for any point $p \in X$, the conjugacy class of $w_n$ is not in $[\pi_1(X-p),[\pi_1(X-p),\pi_1(X-p)]]$.
By Gauss--Bonnet, $\area(S_{m,1}) = 2\pi\cdot(3m-1)$ and $\area(X) = 2\pi\cdot 2n$.
Hence $3m-1\ge 2n$ and therefore $m\ge 2n/3$, as claimed.
\end{proof}

On the other hand, $\sl(w_n|\gamma_3)\le 1$ for all $n$, by Proposition~\ref{algebraic_gamma_3}.
This example shows that the comparison theorem (Theorem~\ref{comparison_theorem}) has no analogue for unstable 
$\gamma_n$-lengths, even if $n=3$.

\end{document}